\documentclass[11pt,a4paper]{amsart}

\usepackage{times,epsfig}
\usepackage{amsfonts,amscd,amsmath,amssymb,amsthm}

\usepackage{bbm}
\numberwithin{equation}{section}

\DeclareSymbolFont{SY}{U}{psy}{m}{n}
\DeclareMathSymbol{\emptyset}{\mathord}{SY}{'306}

\DeclareMathSymbol{\newtimes}{\mathbin}{SY}{'264}

\newcommand{\R}{\mathbb{R}}
\newcommand{\T}{\mathbb{T}}
\newcommand{\C}{\mathbb{C}}

\newcommand{\N}{\mathbb{N}}

\newcommand{\cA}{{\mathcal A}}
\newcommand{\cB}{{\mathcal B}}

\newcommand{\cF}{{\mathcal F}}
\newcommand{\cH}{{\mathcal H}}

\newcommand{\cK}{{\mathcal K}}
\newcommand{\cL}{{\mathcal L}}

\newcommand{\cT}{{\mathcal T}}
\newcommand{\cU}{{\mathcal U}}

\newcommand{\fo}{F{\o}lner sequence~}
\newcommand{\fos}{F{\o}lner sequences~}

\newcommand{\1}{\mathbbm 1}

{\bf}{\it}
{\bf}{\it}
{\bf}{\it}
{\bf}{\it}

{\bf}{\it}

\newtheorem{theorem}{Theorem}[section]{\bf}{\it}
\newtheorem{proposition}[theorem]{Proposition}{\bf}{\it}
\newtheorem{corollary}[theorem]{Corollary}{\bf}{\it}
{\it}{\rm}
{\bf}{\it}
\newtheorem{remark}[theorem]{Remark}{\it}{\rm}
{\bf}{\it}
{\bf}{\it}
{\bf}{\it}

\theoremstyle{definition}
\newtheorem{definition}[theorem]{Definition}

\DeclareMathAlphabet{\Ma}{U}{msa}{m}{n}
\DeclareMathAlphabet{\Mb}{U}{msb}{m}{n}

\DeclareMathAlphabet{\Meuf}{U}{euf}{m}{n}
\DeclareSymbolFont{ASMa}{U}{msa}{m}{n}
\DeclareSymbolFont{ASMb}{U}{msb}{m}{n}
\DeclareMathSymbol{\hrist}{\mathord}{ASMa}{"16}
\DeclareMathSymbol{\varkappa}{\mathalpha}{ASMb}{"7B}
\DeclareMathSymbol{\CrPr}{\mathord}{ASMb}{"6F}

\def\got#1{\Meuf{#1}}
\def\ot #1.{{\got{#1}}}

\title[F\o lner C*-algebras]{Amenable traces and F\o lner C*-algebras}

\author{Pere Ara}
\address{Department of Mathematics,
  Universitat Aut\`onoma de Barcelona, 08193 Bellaterra (Bar\-ce\-lona), Spain}
\email{para@mat.uab.cat}

\author{Fernando Lled\'o}
\address{Department of Mathematics, University Carlos~III Madrid,
  Avda.~de la Universidad~30, 28911 Legan\'es (Madrid), Spain
  and Instituto de Ciencias Matem\'{a}ticas (CSIC - UAM - UC3M - UCM).}
\email{flledo@math.uc3m.es}

\date{\today}

\thanks{The first-named author was partially supported by DGI MICIIN
MTM2011-28992-C02-01, and by the Comissionat per Universitats i
Recerca de la Generalitat de Catalunya. The second-named author was
partially supported by projects
DGI MICIIN MTM2012-36372-C03-01 and Severo Ochoa SEV-2011-0087
of the spanish Ministry of Economy and Competition.}


\begin{document}

\begin{abstract}
In the present article we review an approximation procedure for
amenable traces on unital and separable C*-algebras acting on a
Hilbert space in terms of F\o lner sequences of non-zero finite rank
projections. We apply this method to improve spectral approximation
results due to Arveson and B\'edos. We also present an abstract
characterization in terms of unital completely positive maps of
unital separable C*-algebras admitting a non-degenerate
representation which has a F\o lner sequence or, equivalently, an
amenable trace. This is analogous to Voiculescu's abstract
characterization of quasidiagonal C*-algebras. We define F\o lner
C*-algebras as those unital separable C*-algebras that satisfy these
equivalent conditions. Finally we also mention some permanence
properties related to these algebras.
\end{abstract}

\maketitle

\tableofcontents

\section{Introduction}\label{sec:intro}

There are two well-known important characterizations of discrete
amenable groups, one given in terms of the existence of an invariant
mean, and the other in terms of the existence of F\o lner nets of
finite subsets of the group. In
his seminal article \cite[Section~V]{Connes76}, Alain Connes gave an
analogue of these in the context of von Neumann algebras
introducing amenable traces and F\o lner nets for operators, 
respectively (see also \cite{ConnesIn76,Popa86,PopaIn88}
as well as Sections~\ref{sec:foelner} and \ref{sec:amenable} for precise
definitions and additional results). F\o lner nets for operators are given in terms of
non-zero finite rank orthogonal projections $\{P_\lambda\}_\lambda$ in the
corresponding Hilbert space and satisfying natural approximation
conditions (see Definition~\ref{def:Foelner} for details).
Connes used these concepts as a crucial tool in the
classification of injective type~II$_1$ factors. Recently, this
circle of ideas has been used to define a new invariant for a
general separable type~II$_1$ factor that measures how badly the
factor fails to satisfy Connes' F\o lner type condition
(cf.~\cite{Bannon07}).

In addition to these theoretical developments, \fos
for operators have been also used
in spectral approximation problems: given a sequence of linear
operators $\{T_n\}_{n\in\N}$ acting on a complex Hilbert space $\cH$
that approximates an operator $T$ in a suitable sense, it is natural
to ask how the spectral objects of $T$ relate to those of $T_n$
when $n\to\infty$. We recall next
the following classical approximation result for scalar spectral measures of
Toeplitz operators due to Szeg\"o:
denote by $\T$ the unit circle with normalized Haar measure $d\theta$
and consider the real-valued functions $g$ in $L^\infty(\T)$
which can be thought as (selfadjoint) multiplication operators on the
complex Hilbert space $\cH:=L^2(\T)$,
i.e., $M_g\,\varphi=g\;\varphi$, $\varphi\in\cH$. Denote by $P_n$ the finite-rank
orthogonal projection onto the linear span of $\{z^l\mid z\in\T, l=0,\dots,n\}$
and let $M_g^{(n)}:=P_n\,M_g\,P_n$ be the corresponding finite section matrix. Write the corresponding
eigenvalues (repeated according to multiplicity) as $\{\lambda_{0,n},\dots,\lambda_{n,n}\}$.
Then, for any continuous $f\colon\R\to\R$ one has
\begin{equation}\label{szego}
 \lim_{n\to\infty}
  \frac{1}{n+1}
  \Big( f(\lambda_{0,n})+\dots+f(\lambda_{n,n})\Big) =\int_\T f(g(\theta)) \, d\theta
\end{equation}
(see \cite[Section~8]{Szego20}, \cite[Chapter~5]{bGrenander84} and
\cite{WidomIn65} for a careful analysis of this result; a recent
standard book analyzing many aspects of Toeplitz operators and
containing a large number of references is \cite{bBoettcher06}). The
equation (\ref{szego}) may be also reformulated in terms of
weak-* convergence of the corresponding spectral measures and it
allows the numerical approximation of the spectrum of $M_g$ in terms of
the eigenvalues of its finite sections (see \cite{ArvesonIn94} as
well as Chapter~7 in \cite{Roch08} and references cited therein).
These classical approximation results motivated Arveson to consider
spectral approximations in a more general context than Toeplitz
operators and uses techniques from operator algebras. Among other
results, Arveson gave conditions that guarantee that the essential
spectrum of a selfadjoint operator $T$ may be recovered from the
sequence of eigenvalues of certain finite dimensional compressions
$T_n$ (cf.~\cite{Arveson94,ArvesonIn94}). These results were then
extended by B\'edos who systematically applied the concept of \fo to
spectral approximation problems \cite{Bedos94,Bedos95,Bedos97} (see
also \cite{pLledo11} and references therein). In general, operator
algebraic techniques have also contributed to address these
approximation problems (some examples are
\cite{Brown06,bHagen01,Hansen08}). The notion of F\o lner sequences
has turned to be also interesting in the context of single operator
theory. In \cite{pLledoYakubovich12} Yakubovich and the second-named
author show that several classes of non-normal operators have a 
F\o lner sequence and analyze the relation to the class of finite
operators introduced by Williams in \cite{Williams70}. See also 
\cite{pALlY03} for a review of the notion of F\o lner sequence of projections
in operator algebras and operator theory that includes a new proof that
any essentially normal operator has an increasing F\o lner sequence 
strongly converging to $\1$. In addition,
F\o lner sequences in operator algebras are also important in the
study of growth conditions, i.e., conditions involving the asymptotic growth, as $n\to \infty$, of 
the spaces linearly  spanned by the sets $X^n$, where $X$ is a finite generating set of  
a C*-algebra  (see, e.g., \cite{Vaillant96,CS08}). We
also refer to \cite{Brown04,Brown06a,Ozawa04} for a thorough description of
the relations of amenable traces and F\o lner sequences to other
important areas like, e.g., Connes' embedding problem.

An important step in the proof of the Arveson-B\'edos spectral approximation
result mentioned above is the compatibility between the choice of
the F\o lner sequence in the Hilbert space and the amenable trace.
In fact, if the unital and separable concrete
C*-algebra $\cA\subset\cL(\cH)$ has an amenable trace $\tau$
and $\{P_n\}_n$ is a F\o lner sequence of non-zero
finite rank projections for $\cA$ it is needed that
the projections approximate the amenable trace in the following
natural sense
\begin{equation}\label{eq:approx-trace}
 \tau (A)=\lim_{n\to\infty} \frac{{\mathrm{Tr}}(AP_n)}{\mathrm{Tr}(P_n)}\;,\quad A\in \cA\;,
\end{equation}
where $\mathrm{Tr}(\cdot)$ denotes the canonical trace on
$\cL(\cH)$. Now given $\cA\subset\cL(\cH)$ with an amenable trace
$\tau$ it is possible to construct a F\o lner sequence in different
ways. As observed by B\'edos in \cite{Bedos95} one way to obtain a
F\o lner sequence $\{P_n\}$ for $\cA\subset\cL (\cH )$ is
essentially contained in \cite{Connes76,ConnesIn76}. In these
articles Connes adapts the group theoretic methods by Day and
Namioka to the context of operators. Using this technique one loses
track of the initial amenable trace $\tau$, in the sense that the
sequence $\{ P_n\}$ does not necessarily satisfy
(\ref{eq:approx-trace}). To avoid this problem one may assume in
addition that $\cA$ has a unique tracial state. This is sufficient
to guarantee a good spectral approximation behavior of relevant
examples like almost Mathieu operators, which are contained in the
irrational rotation algebra (cf.~\cite{bBoca01}).

In contrast with the previous method, the construction of a F\o lner
sequence given in \cite[Theorem 6.1]{Ozawa04} (see also
\cite[Theorem 6.2.7]{bBrown08}) allows the approximation of the original
trace as in Eq.~(\ref{eq:approx-trace}). We will review this method
in Section~\ref{sec:amenable} and apply it to prove a spectral
approximation result in the spirit of Arveson and B\'edos, but
removing the hypothesis of a unique trace (see
Theorem~\ref{teo:apply} for details as well as
\cite[p.~354]{Arveson94}, \cite[Theorem~1.3]{Bedos95} or
\cite[Theorem~6~(iii)]{Bedos97}).

In the last section of this article we will also give an abstract
characterization of unital separable C*-algebras admitting a
non-degenerate representation $\pi$ on a Hilbert space such that
there is a F\o lner sequence for $\pi(\cA)$ or, equivalently, such
that $\pi(\cA)$ has an amenable trace (see Theorem
\ref{thm:characFolner}). More precisely, we obtain that these
conditions are equivalent to the existence of a sequence of unital
completely positive (u.c.p.) maps $\varphi_n \colon \cA \to
M_{k(n)}(\C)$ which is asymptotically multiplicative with respect to
the normalized Hilbert-Schmidt norm $\| \cdot \|_{2, \text{tr}}$ on
$M_{k(n)}(\C)$. Motivated by this relationship, we call the
C*-algebras admitting such  finite dimensional approximations {\em
F\o lner C*-algebras} (Definition \ref{def:FA}). It turns out that
this is the same class as the {\em weakly hypertracial} C*-algebras
studied by B\'edos in \cite{Bedos95}. Our result is inspired by
Voiculescu's abstract characterization of quasidiagonal C*-algebras
(cf.~\cite{voicu91}), which asserts that a unital separable
C*-algebra $\cA$ is quasidiagonal if and only if there is a sequence
of u.c.p. maps $\varphi_n \colon \cA \to M_{k(n)}(\C)$ which is
asymptotically multiplicative and asymptotically isometric with
respect to the operator norm on $M_{k(n)}(\C)$.  We end the paper by
recalling some known permanence properties of weakly hypertracial
C*-algebras, proved by B\'edos in \cite{Bedos95}.

\smallskip

We refer the reader to \cite{bBrown08}, especially Chapter 6, for background 
material. The paper is basically self-contained modulo some of the results contained
in that book.

\smallskip

{\bf Notation:}
We will denote by $\cL(\cH)$ the C*-algebra of bounded linear
operators on the complex separable Hilbert space $\cH$, and by
$\cK(\cH)$ the ideal of compact operators on $\cH$. The unitary
group of a unital C*-algebra $\cA$ is denoted by $\cU(\cA)$. We will assume
that any representation of a unital C*-algebra preserves the unit
(i.e. it is non-degenerate). To simplify expressions we will sometime
use the standard notation for the commutator of two operators: $[A,B]:=AB-BA$.

\section{F\o lner type conditions for operators}\label{sec:foelner}
The notion of F\o lner sequences for operators has its origins in
group theory. Recall that a discrete countable group $\Gamma$ is
amenable if it has an invariant mean, i.e. there is a positive
linear functional $\psi$ on $\ell^\infty(\Gamma)$ with norm one such
that
\[
  \psi(\gamma f)=\psi(f)\;,\quad \gamma\in\Gamma\;,\quad f\in \ell^\infty(\Gamma)\;,
\]
where $(\gamma f)(\gamma_0):=f(\gamma^{-1}\gamma_0)$.
A F{\o}lner sequence for
$\Gamma$ is a sequence of non-empty finite subsets $\Gamma_i\subset\Gamma$
that satisfy
\begin{equation}\label{eq:foelner-g}
\lim_{i}
\frac{|(\gamma \Gamma_i)\triangle
\Gamma_i|}{|\Gamma_i|} =0\qquad\text{for all}\quad \gamma\in\Gamma \,,
\end{equation}
where $\triangle$ denotes the symmetric difference and $|X|$ is the
cardinality of $X$ for any set $X$. Then, $\Gamma$ has a F{\o}lner
sequence if and only if $\Gamma$ is amenable (cf.~Chapter~4 in
\cite{bPaterson88}). If $\Gamma$ has a F\o lner sequence one can
always find another F\o lner sequence which, in addition to
Eq.~(\ref{eq:foelner-g}), is also increasing and complete, i.e.
$\Gamma_i\subset\Gamma_j$ if $i\leq j$ and $\Gamma=\cup_i \Gamma_i$.

The counterpart of the previous definition in the context of operator
algebras is given as follows:
\begin{definition}\label{def:Foelner}
Let $\mathcal{A}\subset\cL(\cH)$ be a C*-algebra of bounded
operators on a complex separable Hilbert space $\mathcal{H}$.
\begin{itemize}
 \item[(i)] A sequence
of non-zero finite rank orthogonal
projections $\{P_n\}_{n\in \N}\subset\cL(\cH)$
is called a {\em F{\o}lner sequence for $\mathcal{A}$} if
\begin{equation}\label{eq:F1}
\lim_{n} \frac{\|A P_n-P_n A\|_2}{\|P_n\|_2} = 0\;\;,\quad A\in\cA\;,
\end{equation}
where $\|\cdot\|_2$ denotes the Hilbert-Schmidt norm.

\noindent The F\o lner sequence $\{ P_n\}_n$ is said to be a {\em
proper} F\o lner sequence if it is an increasing sequence of
projections converging to $\1$ in the strong operator topology.
\item[(ii)] $\cA$ satisfies the {\em F\o lner condition} if for any finite
set $\cF\subset\cA$ and any $\varepsilon >0$ there exists a finite rank orthogonal
projection $P$ such that
\begin{equation}\label{eq:FC}
 \frac{\|A P-P A\|_2}{\|P\|_2} < \varepsilon \;\;,\quad A\in\cF\;.
\end{equation}

\end{itemize}
\end{definition}

We will state next some immediate consequences of the definition that will
be used later on. Note that the definition of F\o lner sequence can extended in 
an obvious way to any set $\cT$ of bounded operators in $\cL(\cH)$ requiring 
condition (\ref{eq:F1}) for all $A\in\cT$. (See also 
Sections~1 and 2 in \cite{pLledoYakubovich12}).

\begin{proposition}\label{total}
Let $\cT\subset\cL(\cH)$ be a set of operators
and $\{P_n\}_{n\in \N}$ a sequence of non-zero finite rank
orthogonal projections.
\begin{itemize}
\item[(i)] $\{P_n\}_{n\in \N}$ is a F{\o}lner sequence
for $\cT$ if and only if it is a F{\o}lner sequence for $C^*(\cT,\1)$
(the C*-algebra generated by $\cT$ and $\1$).
\item[(ii)] Let $\cT$ be a selfadjoint set (i.e.~$\cT^*=\cT$). Then
$\{P_n\}_{n\in \N}$ is a F{\o}lner sequence for $\cA$ if and only if
one of the four following equivalent conditions holds for all
$A\in\cA$:
\begin{equation}\label{F1}
\lim_{n} \frac{\|A P_n-P_n A\|_p}{\|P_n\|_p} = 0,\qquad
p\in\{1,2\}
\end{equation}
or
\begin{equation}\label{F2}
\lim_{n} \frac{\|(I-P_n)A P_n\|_p}{\|P_n\|_p} = 0,\qquad
p\in\{1,2\}\;,
\end{equation}
where $\|\cdot\|_1$ and $\|\cdot\|_2$ are the trace-class and Hilbert-Schmidt
norms, respectively.
\end{itemize}
\end{proposition}
\begin{proof}
Part (i) is straightforward and part (ii) is Lemma~1 in \cite{Bedos97}.
\end{proof}

The following proposition is shown by a standard argument.

\begin{proposition}\label{pro:equivalent}
Let $\mathcal{A}\subset\cL(\cH)$ be a separable C*-algebra. Then,
$\mathcal{A}$ has a F\o lner sequence if and only if $\mathcal{A}$
satisfies the F\o lner condition.
\end{proposition}

\subsection{Quasidiagonality}
The existence of a \fo for a set of operators $\mathcal{T}$ is a
weaker notion than quasidiagonality. Recall that a (separable) set
of operators $\mathcal{T}\subset\mathcal{L}(\mathcal{H})$ is said to
be quasidiagonal if there exists an increasing sequence of finite-rank
projections $\{P_n\}_{n\in \N}$ converging strongly to $\1$ and such that
\begin{equation}\label{QD}
\lim_{n}\|T P_n-P_n T\|=0\;,\quad T\in\mathcal{T}\;.
 \end{equation}
The existence of proper F\o lner sequences can be understood as a
quasidiagonality condition, but relative to the growth of the
dimension of the underlying spaces. It can be easily shown that if
$\{P_n\}_n$ quasidiagonalizes a family of operators $\cT$, then this
sequence of non-zero finite rank orthogonal projections is also a
F\o lner sequence for $\cT$. In \cite{voicu91}, Voiculescu
characterized abstractly quasidiagonality for unital separable
C*-algebras in terms of u.c.p. maps (see also \cite{Voiculescu93}).
This has become by now the
standard definition of quasidiagonality for operator algebras (see,
for example, \cite[Definition~7.1.1]{bBrown08}):

\begin{definition}
\label{def:abstractqd} A unital separable C*-algebra $\cA$ is called
{\em quasidiagonal} if there exists a sequence of u.c.p. maps
$\varphi _n \colon \cA\to M_{k(n)}(\mathbb C)$ which is both
asymptotically multiplicative (i.e. $\| \varphi_n (AB) -\varphi_n
(A)\varphi_n (B) \| \to 0$ for all $A,B\in \cA$) and asymptotically
isometric (i.e. $\| A \| =\lim _{n\to \infty} \| \varphi_n (A) \| $
for all $A\in \cA$).
\end{definition}

The unilateral shift is a prototype that shows the difference
between the notions of F\o lner sequences and quasidiagonality. On
the one hand, it is a well-known fact that the unilateral shift $S$
is not a quasidiagonal operator. (This was shown by Halmos in
\cite{Halmos68}; in fact, in this reference it is shown that $S$ is
not even quasi-triangular.) In the setting of abstract C*-algebras
it can also be shown that a C*-algebra containing a proper
(i.e.~non-unitary) isometry is not quasidiagonal (see,
e.g.~\cite{BrownIn04,bBrown08}). It can be shown, though, that
certain weighted shifts are quasidiagonal (cf.~\cite{Smucker82}).

On the other hand, it is easy to give a F\o lner sequence for $S$. In fact, define
$S$ on $\cH:=\ell^2(\N_0)$ by $Se_i:=e_{i+1}$, where
$\{e_i\mid i=0,1,2,\dots\}$ is the canonical basis of $\cH$
and consider for any $n$ the orthogonal projections $P_n$ onto span$\{e_i\mid i=0,1,2,\dots, n\}$.
Then
\[
 \big\|[P_n,S]\big\|_2^2=\sum_{i=1}^\infty \Big\|[P_n,S]e_i\Big\|^2=\|e_{n+1}\|^2=1
\]
and
\[
\frac{\big\|[P_n,S]\big\|_2}{\| P_n\|_2}=\frac{1}{\sqrt{n+1}}\;\mathop{\longrightarrow}\limits_{n\to\infty} \;0\;.
\]

\section{Approximations of amenable traces}\label{sec:amenable}

The existence of F\o lner sequences for a concrete C*-algebra $\cA$ has several
operator algebraic consequences. The most prominent one is the existence of an
amenable trace on $\cA$. In this section we will review a particularly useful approximation
of an amenable trace and consider its application to a spectral approximation
problem.


We refer to Chapter~3 of \cite{Brown06a} for a careful analysis of subsets of the 
amenable traces that can be related to suitable finite-dimensional approximation
properties of the corresponding C*-algebra.
Fundamental results related to amenable traces were obtained by Kirchberg (using the name
liftable tracial state) in \cite{Kirchberg94}.

Let $\cA\subset\cL(\cH)$ be a unital C*-algebra. A state $\tau$ on $\cA$
is called an {\em amenable trace} if there exists a state $\psi$ on
$\cL(\cH)$ such that $\psi |_\cA=\tau$ and
\begin{equation*}
\psi(X A) = \psi(A X)\;,\quad X\in \cL(\cH)\;,\;A\in\cA\,.
\end{equation*}
The state $\psi$ is also referred to in the literature as a
hypertrace on $\cL(\cH)$. Amenable traces are the operator algebraic
analogues of the invariant means for groups mentioned at the
beginning of the preceding section
(cf.~\cite{Connes76,ConnesIn76,Bedos97,Brown06a}). Part (ii) of the
following result is known to experts (see e.g. Exercise~6.2.6 in
\cite{bBrown08}); part (i) is well known and stated in several
places in the literature. Since the result is very important for
this paper, and for convenience of the reader, we give a complete
proof of it.

\begin{proposition}\label{exercise}
Let $\cA\subset\cL(\cH)$ be a unital separable C*-algebra.
\begin{itemize}
\item[(i)] If $\cA$ has a F\o lner sequence
$\{P_n\}_n$, then $\cA$ has an amenable trace.

\item[(ii)]
Assume that $\cA\cap\cK(\cH)=\{0\}$, and let $\tau$ be an amenable
trace on $\cA$. Then $\cA$ has a F\o lner sequence $\{P_n\}_n$
satisfying
\begin{equation}\label{eq:conv-trace}
\tau(A)=\lim_{n\to\infty}\frac{\mathrm{Tr}(AP_n)}{\mathrm{Tr}(P_n)}\;,\quad A\in\cA \;,
\end{equation}
where $\mathrm{Tr}$ denotes the canonical trace on $\cL(\cH)$.
\end{itemize}
\end{proposition}

\begin{proof}
The proof of part (i) is a standard argument. Let $\{P_n\}_n$ be a
F\o lner sequence for $\cA$. Consider the following canonical
sequence of states of $\cL(\cH)$
\[
 \psi_n(X):=\frac{\mathrm{Tr}(XP_n)}{\mathrm{Tr}(P_n)}\;,\quad X\in\cL(\cH)\;.
\]
Using Eq.~(\ref{F1}) with $p=1$ it is an
$\frac{\varepsilon}{2}$-argument to show that any weak-* cluster
point of the sequence $\psi_n$ (which exists by weak-* compactness)
defines a hypertrace on $\cA$.

Part~(ii) requires several steps. It is enough to show that
for any finite selfadjoint set 
$\cF\subset\cA$ and any $1> \varepsilon>0$ there exists a finite
rank orthogonal projection $Q\in\cL(\cH)$ such that
\begin{equation}\label{eq:Q}
 \frac{\|B Q-Q B\|_2}{\|Q\|_2} < \varepsilon
 \quad\mathrm{and}\quad
 \left|\tau(B)-\frac{\mathrm{Tr}(B Q)}{\mathrm{Tr(Q)}}\right| < \varepsilon
 \;\;,\quad B\in\cF
\end{equation}
(cf.~Definition~\ref{def:Foelner} and Proposition~\ref{pro:equivalent}).

Let $\cF\subset\cA$ and $1> \varepsilon>0$ be given as before.
First, from Stinespring's theorem and the proof of Theorem~6.2.7 in
\cite{bBrown08} there exists a u.c.p. map
\[
 \varphi\colon\cA\to M_k(\C)
\]
(where $M_k(\C)\cong\cL(\cH_k)$ and $\mathrm{dim}\cH_k=k$),
an isometry
\[
 V\colon\cH_k\to\Meuf{h}
\]
and a representation
\[
 \pi\colon\cA\to\cL(\Meuf{h})
\]
satisfying
\begin{eqnarray}
 \varphi(A)&=& V^*\pi(A) V\;,\quad\;\;A\in\cA\;. \label{eq:VAV} \\[3mm]
|\mathrm{tr}\left(\varphi(B^*B)-\varphi(B^*)\varphi(B)\right)|
           &<&\varepsilon\;, \quad\;\;B\in\cF \;. \label{eq:approxvarphi}\\[3mm]
 \left|\tau(B)-\mathrm{tr}(\varphi(B))\right|
           &<& \varepsilon\;,\quad \;\;B\in\cF\;,\label{eq:3}
\end{eqnarray}
where $\mathrm{tr}(\cdot)$ is the unique tracial state on the matrix algebra. We introduce next Stinespring's projection
\[
 P:=VV^*\;,
\]
which is a finite rank projection in $\cL(\Meuf{h})$. Using the relation
$P\pi(A)P=V\varphi(A)V^*$, $A\in\cA$, it is straightforward to show that
\begin{equation}\label{eq:off-block}
 \frac{\|(\1-P)\pi(B)P\|_2}{\|P\|_2} < \sqrt{\varepsilon} \;\;,\quad B\in\cF\;.
\end{equation}

The second step in the proof makes use of Voiculescu's theorem as
stated, e.g., in \cite[\S~1.7]{bBrown08}. Consider the inclusion
$\iota\colon\cA\to\cL(\cH)$. Since $\cA\cap\cK(\cH)=\{0\}$ we have
that $\iota$ and $\iota\oplus\pi$ are approximately unitarily
equivalent relative to compacts. In particular, there is a unitary
$W\colon\cH\to\cH\oplus\Meuf{h}$ such that
\begin{equation}\label{eq:Voicu}
 \left\|B-W^*(B\oplus \pi(B))W\right\|<\varepsilon\;,\quad B\in\cF\;.
\end{equation}
Define the orthogonal projection $Q$ on $\cH$ by
\[
 Q:=W^* (0\oplus P) W
\]
and note that $\|Q\|_2=\|P\|_2$, where the Hilbert-Schmidt norms are
considered on the Hilbert spaces $\cH$ and $\Meuf{h}$, respectively.
Putting $Q^\perp:=\1-Q$ and using Eqs.~(\ref{eq:off-block}) and
(\ref{eq:Voicu}) we have the following estimates for any $B\in\cF$:
\begin{eqnarray*}
\|Q^\perp BQ\|_2  &\leq& \|Q^\perp (B-W^*(B\oplus \pi(B))W) Q\|_2  
                  + \|Q^\perp (W^*(B\oplus \pi(B))W) Q\|_2\\[3mm]
                &\leq& \|B-W^*(B\oplus \pi(B))W\|\;\|Q\|_2 +
                       \| W^*(\1\oplus P^\perp)(B\oplus \pi(B))(0\oplus P) W)\|_2\\[3mm]
                &\leq& \varepsilon \;\|Q\|_2 + \|P^\perp \,\pi(B)\, P\|_2 \\[3mm]
                &\leq& 2\sqrt{\varepsilon} \;\|Q\|_2  \;.
\end{eqnarray*}
Since $\cF ^*=\cF$, we obtain the first condition of
Eq.~(\ref{eq:Q}) (with $4\sqrt{\varepsilon}$ instead of $\varepsilon $).

We still have to show the second condition in Eq.~(\ref{eq:Q}). Note
that for any $B\in\cF$ we have
\begin{eqnarray*}
 \mathrm{tr}(\varphi(B))&=& \mathrm{tr}(V^*\pi(B)V)\;\;=\;\;\frac{\mathrm{Tr}(P\pi(B))}{\mathrm{Tr}(P)}\\[3mm]
                        &=& \frac{\mathrm{Tr}((0\oplus P)(B\oplus\pi(B)))}{\mathrm{Tr}(P)}\\[3mm]
                        &=& \frac{\mathrm{Tr}(Q(W^* (B\oplus\pi(B)) W))}{\mathrm{Tr}(Q)}\;.
\end{eqnarray*}
Finally, we can use the previous relation as well as Eqs.~(\ref{eq:3}) and (\ref{eq:Voicu}) to show the estimates
\begin{eqnarray*}
\left|\tau(B)-\frac{\mathrm{Tr}(Q B)}{\mathrm{Tr(Q)}}\right|
           &\leq& |\tau(B)-\mathrm{tr}(\varphi(B))|+\left|\mathrm{tr}(\varphi(B)) -\frac{\mathrm{Tr}(Q B)}{\mathrm{Tr(Q)}}\right| \\[3mm]
           &\leq& \varepsilon + \left|
                                \frac{\mathrm{Tr}\Big(Q \left(W^*( B\oplus\pi(B)  )W- B\right)\Big)}{\mathrm{Tr(Q)}}
                                \right|\\[3mm]
           &\leq& \varepsilon + \|W^*( B\oplus\pi(B)  )W- B \| \;\;<\;\; 2\varepsilon\;,
\end{eqnarray*}
and the proof is concluded.
\end{proof}

\subsection{Approximation of spectral measures} \label{subsec:spec}

We will now present an application of
Proposition~\ref{exercise}~(ii) to spectral approximation. The
argument in the proof of Theorem~\ref{teo:apply} below will be used
later in the proof of our main characterization result (Theorem
\ref{thm:characFolner}).

We need to recall from \cite{Bedos97} the definition of Szeg\"o
pairs for a concrete C*-algebra $\cA\subset\cL(\cH)$. This notion
incorporates the good spectral approximation behavior of scalar
spectral measures of selfadjoint elements in $\cA$ and is motivated
by Szeg\"o's classical approximation results mentioned in the introduction.

Let $\cA$ be a unital C*-algebra acting on $\cH$ and let $\tau$ be a
tracial state on $\cA$. For any self\-adjoint element $T\in\cA$ we
denote by $\mu_T$ the spectral measure associated with the trace
$\tau$ of $\cA$. Consider a sequence $\{P_n\}_n$ of non-zero finite
rank projections on $\cH$ and write the corresponding (selfadjoint)
compressions as $T_n:=P_n T P_n$. Denote by $\mu_T^n$ the
probability measure on $\R$ supported on the spectrum of $T_n$,
i.e.,
\[
 \mu_T^n(\Delta):=\frac{N_T^n(\Delta)}{d_n}\;,\quad \Delta\subset\R\quad \mathrm{Borel}\;,
\]
where $N_T^n(\Delta)$ is the number of eigenvalues of $T_n$
(multiplicities counted) contained in $\Delta$ and
$d_n$ is the dimension of the subspace $P_n\cH$.

We say that $\left\{
\{P_n\}_n \,,\,\tau\right\}$ is a {\it Szeg\"o pair} for $\cA$ if
$\mu_T^n\to \mu_T$ weakly for all selfadjoint elements $T\in\cA$,
i.e.,
\[
 \lim_{n\to\infty}
  \frac{1}{d_n}
  \Big( f(\lambda_{1,n})+\dots+f(\lambda_{d_n,n})\Big) =\int f(\lambda) \, d\mu_T(\lambda)
  \;,\quad f\in C_0(\R)  \;,
\]
where $\{\lambda_{1,n},\dots,\lambda_{d_n,n}\}$ are the eigenvalues
(repeated according to multiplicity) of $T_n$.

By \cite[Theorem 6~(i),(ii)]{Bedos97}, if $\left\{ \{P_n\}_n
\,,\,\tau\right\}$ is a Szeg\"o pair for $\cA$, then $\{P_n\}_n$
must be a F\o lner sequence for $\cA$, $\tau$ must be an amenable
trace, and equation (\ref{eq:conv-trace}) must hold for every $A\in
\cA$. Proposition~\ref{exercise}~(ii) allows to complete {\it any}
amenable trace $\tau$ on $\cA$ with a F\o lner sequence so that the
pair $\left\{ \{P_n\}_n \,,\,\tau\right\}$ is a {\it Szeg\"o pair}
for $\cA$, as follows.

\begin{theorem}\label{teo:apply}
Let $\cA$ be a unital, separable C*-algebra acting on a separable Hilbert
space $\cH$, and assume that $\cA \cap \cK (\cH)= \{ 0 \}$.  If
$\tau$ is an amenable trace on $\cA$, then there exists a proper
F\o lner sequence $\{P_n\}_n$ such that $\left\{ \{P_n\}_n
\,,\,\tau\right\}$ is a Szeg\"o pair for $\cA$.
\end{theorem}
\begin{proof}
 By using the same arguments as in the proof of Proposition
\ref{exercise}(ii), we get that the following local condition is
satisfied: For every finite selfadjoint set $\cF$ of $\cA$, and for
every $\varepsilon >0$, there exists a finite rank orthogonal
projection $Q\in \cL (\cH)$ such that
\[
\frac{\| [Q, A] \|_2}{\|Q\|_2} <\varepsilon \quad\mathrm{and}\quad
\left|\tau(A)-\frac{\mathrm{Tr}(Q A)}{\mathrm{Tr(Q)}}\right|<\varepsilon
\quad\mathrm{for~all}\quad
A\in\cF\;.
\]

>From this local condition, we are going to construct an increasing sequence
$\{P_n\}_n$ such that $P_n\nearrow \1$ in the strong operator topology and
such that
\[
\lim _n \frac{\| [P_n,A]\|_2}{ \| P_n\|_2} = 0  \;, \qquad
\tau(A)= \lim _n \frac{\mathrm{Tr}(P_n A)}{\mathrm{Tr}(P_n)}
\quad\mathrm{for~all}\quad
A\in\cA\;.
\]
Take a countable dense subset $\{A_1,A_2,\dots
\}$ of $\cA$, with $A_i\ne 0$ for all $i$. Take $\varepsilon _n=
2^{-n}$ for all $n\ge 1$ and let $Q_n$ be a finite rank orthogonal
projection such that
\[
 \frac{\| [Q_n, A_i]\|_2 }{\| Q_n\|_2} < \varepsilon _n
  \quad\mathrm{and}\quad
 \left|\tau(A_i)-\frac{\mathrm{Tr}(Q_n A_i)}{\mathrm{Tr}(Q_n)}\right|<\varepsilon_n
   \quad\mathrm{for}\quad  i=1,\dots ,n .
\]
We will show next that we may also assume that
\[
\text{dim} \big(Q_n(\cH )\big) \underset{n\to \infty}{\longrightarrow} \infty.
\]
In fact, recall from the proof of Proposition~\ref{exercise}~(ii)
that the dimension of $Q_n$ coincides with the dimension of Stinespring's projection
associated to the corresponding u.c.p. map
$\varphi_n\colon \cA \to M_{k(n)}(\C )$. Since we can replace $\varphi_n$
with a finite direct sum of $n$ copies of $\varphi_n$, without changing the fundamental
estimates (\ref{eq:approxvarphi}) and (\ref{eq:3}), we obtain our claim.

Now consider a sequence $\{ R_n\}_n$ of finite-rank orthogonal
projections such that $R_n\nearrow \1$. Take $P_1=Q_1$, $R_1=Q_1$ and
assume that $P_1,\dots ,P_n$ have been constructed so that the following conditions hold:
\begin{enumerate}
\item $R_i\le P_i$ for $i=1,\dots ,n$.\vspace{1mm}
\item $P_1\le P_2\le \cdots \le P_n$. \vspace{1mm}
\item $\| [P_i,A_j ]\|_2 <\varepsilon _i\|P_i \|_2$ for $1\le j\le i\le n $.\vspace{3mm}
\item  $\left|\tau(A_i)-\frac{\mathrm{Tr}(P_n A_i)}{\mathrm{Tr}(P_n)}\right|<\varepsilon _n$ for $1\le i\le n$.
\end{enumerate}
Since $\text{dim} \big(Q_l(\cH )\big) \underset{l\to
\infty}{\longrightarrow} \infty$,  we may take $m>n+1$ such
that\footnote{If $P,Q$ are orthogonal projections on $\cH$ we denote
by $P\vee Q$ the orthogonal projection onto the closure of
span$\,\{P\cH\cup Q\cH\}$.}
\begin{equation}\label{eq:Qgrow}
\| Q_m \|_2\geq \frac{4 \, \| R_{n+1}\vee P_n \|_2 }{\varepsilon _{n+1}}\,\text{max}\{ 1,  \| A_1\|, \dots , \|A_{n+1} \|\}\;.
\end{equation}

Set $P_{n+1} := R_{n+1} \vee P_n\vee Q_m $ and we have to show that the corresponding Eqs.~$(1)-(4)$ above are also
true for step $n+1$. Clearly $R_{n+1}\le
P_{n+1}$ and $P_n\le P_{n+1}$. We can write $P_{n+1}= Q_m\oplus
P'_{n+1}$ with $\| P'_{n+1} \|_2 \le \| R_{n+1}\vee P_n \|_2$. For
$i=1,\dots , n+1$, we have
\begin{align*}
\frac{\|[ P_{n+1}, A_i ]\|_2}{ \|P_{n+1} \|_2} & \le \frac{\| [ Q_m,
A_i ]\|_2}{\| P_{n+1} \|_2} +\frac{ \| [P'_{n+1}, A_i]\|_2}{\|
P_{n+1} \|_2} \\[2mm]
& \le \frac{\| [ Q_m, A_i ]\|_2}{\| Q_{m} \|_2} +\frac{2\cdot
\|A_i\| \cdot  \| P'_{n+1}\|_2}{\|
Q_{m} \|_2} \\[2mm]
& < \frac{\varepsilon _{n+1}}{2}+ \frac{\varepsilon _{n+1}}{2}
=\varepsilon _{n+1}\;,
\end{align*}
where for the last estimate we have used (\ref{eq:Qgrow}).

Finally, we still have to show condition (4) that implies $P_{n+1}$
is also a good approximation of the amenable trace.
Write $\alpha := \frac{\mathrm{Tr}(Q_m)}{\mathrm{Tr}(P_{n+1})} < 1$.
Then using again (\ref{eq:Qgrow}) note that
$$| 1-\alpha | = \frac{\| P'_{n+1} \|_2^2}{\| P_{n+1} \|_2^2}\le
\frac{\varepsilon_{n+1}^2}{16} \Big(\text{max}\{ 1,  \| A_1\|^2, \dots ,
\|A_{n+1} \|^2\}\Big)^{-1} .$$

Hence using again the decomposition $P_{n+1}=Q_m\oplus P'_{n+1}$ we have
for $i=1,\dots , n+1$:
\begin{align*}
\left| \tau(A_i)  -\frac{\mathrm{Tr}(A_i
P_{n+1})}{\mathrm{Tr}(P_{n+1})}\right|  & \le \left|
\tau(A_i)-\frac{\mathrm{Tr}(A_i Q_m)}{\mathrm{Tr}(P_{n+1})}\right| +
\frac{ |\mathrm{Tr}(A_iP'_{n+1})|}{\mathrm{Tr}(P_{n+1})} \\[2mm]
& \le  \left|\tau(A_i)-\frac{\mathrm{Tr}(A_i
Q_m)}{\mathrm{Tr}(Q_m)}\right| +\left|
\frac{\mathrm{Tr}(A_iQ_m)}{\mathrm{Tr}(Q_m)}(1- \alpha)\right| +
\frac{|\mathrm{Tr}(A_iP_{n+1}')|}{\mathrm{Tr} (P_{n+1})} \\[2mm]
& \le \varepsilon _m + 2\|A_i\|\cdot
\frac{\varepsilon_{n+1}^2}{16}\cdot \left(\mathop{\mathrm{max}}_i\{ 1,\|A_i\|^2 \}\right)^{-1}
\\
& \le \varepsilon _m + \frac{\varepsilon _{n+1}^2}{8} \\[2mm]
& \le \frac{\varepsilon_{n+1}}{2}+ \frac{\varepsilon_{n+1}^2}{8} <
\varepsilon_{n+1}.
\end{align*}

It follows that $P_n\nearrow \1$ and that $\lim _n \frac{\|[P_n,
A]\|_2}{\|P_n\|_2} = 0$ for all $A\in \cA$.

\smallskip

Now the proof of Theorem~6~(iii) in \cite{Bedos97} gives that
$\left\{ \{ P_n \}, \tau \right\}$ is a Szeg\" o pair for $\cA$.
\end{proof}

\begin{remark}
The preceding theorem is a contribution to the study of Szeg\"o-type
theorems in the context of C*-algebras. Note, nevertheless, that the existence
a F\o lner sequence approximating nicely the amenable trace is  
established in abstract terms. This gives in general 
no clue of what the matrix approximations of concrete
operators are. It would be interesting to construct in concrete cases 
explicit F\o lner sequences of this type to address spectral approximation
problems in this more general context (see, e.g., Chapter~7 in \cite{bHagen01}). 
\end{remark}

\section{F\o lner C*-algebras}\label{sec:FA}

In this section, we introduce the abstract definition of a F\o lner
C*-algebra and we obtain our main result characterizing F\o lner
C*-algebras in terms of F\o lner sequences and also of amenable
traces. Moreover, we state some consequences for tensor products and
nuclear C*-algebras.

We denote by $\mathrm{tr}(\cdot)$ the unique tracial state on a matrix algebra
$M_{n}(\C)$.
\begin{definition}\label{def:FA}
Let $\cA$ be a unital, separable C*-algebra.
\begin{itemize}
 \item[(i)] We say that $\cA$ is a {\em F\o lner C*-algebra} if there exists a sequence of u.c.p. maps
$\varphi_n\colon\cA\to M_{k(n)}(\C)$ such that
\begin{equation}\label{eq:mult-2}
\lim_n\|\varphi_n(AB)-\varphi_n(A)\varphi_n(B)\|_{2,\mathrm{tr}}=0\;,\quad A,B\in\cA\;,
\end{equation}
where $\|F\|_{2,\mathrm{tr}}:=\sqrt{\mathrm{tr}(F^*F)}$, $F\in M_{n}(\C)$\;.
 \item[(ii)]  We say that $\cA$ is a {\em proper F\o lner C*-algebra} if there exists a sequence of u.c.p. maps
  $\varphi_n\colon\cA\to M_{k(n)}(\C)$ satisfying the previous Eq.~(\ref{eq:mult-2}) and which, in addition,
  are asymptotically isometric, i.e.,
\begin{equation}\label{eq:norm}
   \|A\|=\lim_n\|\varphi_n(A)\|\;,\quad A\in\cA\;.
\end{equation}
\end{itemize}
\end{definition}

It is clear that if $\cA$ is a separable, unital and quasidiagonal
C*-algebra (cf.~Definition \ref{def:abstractqd}), then $\cA$ is a
proper F\o lner algebra. 
Moreover, let $\cB$ be a {\em unital} C*-subalgebra of $\cA$.
Clearly, if $\cA$ is a (proper) F\o lner algebra, then $\cB$ is
again a (proper) F\o lner algebra. This is not true if $\cB$ is a
non-unital C*-subalgebra (i.e.~$\1_{\cA}\notin\cB$).

Although, in principle, the two concepts--F\o lner and properly F\o
lner--seem to be different, we can show that they indeed define the
same class of unital, separable C*-algebras:

\begin{proposition}
\label{prop:F=PF} Let $\cA$ be a unital separable C*-algebra. Then
$\cA$ is a F\o lner C*-algebra if and only if $\cA$ is a proper F\o
lner C*-algebra.
\end{proposition}

\begin{proof}
Assume that $\cA$ is a F\o lner C*-algebra, and let $\varphi _n
\colon \cA \to M_{k(n)}(\C )$ be a sequence of u.c.p maps such that
(\ref{eq:mult-2}) holds. Considering the direct sum of a
sufficiently large number of copies of $\varphi_n$, for each $n$, we
may assume that
\begin{equation}
\label{eq:noverk(n)goesto0} \lim _{n\to \infty} \frac{n}{k(n)}= 0.
\end{equation}
Let $\pi \colon \cA \to \cL (\cH )$ be a faithful representation of
$\cA$ on a separable Hilbert space $\cH$. Let $\{ P_n\}_n$ be an
increasing sequence of orthogonal projections on $\cH$, converging
to $\1$ in the strong operator topology and such that $\text{dim}
(P_n(\cH )) =n$ for all $n$. Then for all $A\in \cA$ we have $\| A\|
=\lim _n \| P_n\pi (A) P_n \| $.  Let $\psi _n \colon \cA \to
M_{k(n)+n}(\C) $ be given by:
$$\psi _n (A) = \varphi _n (A) \oplus P_n \pi (A) P_n ,$$
for $A\in \cA$. Then $\psi _n$ is a u.c.p. map. For $A,B\in \cA$,
set $X_n= P_n \pi (A)(1-P_n)\pi (B) P_n $. Then we have
\begin{align*}
\| \psi _n (AB) -\psi _n(A) & \psi _n(B)  \| _{2, \mathrm{tr}}^2 \le
\| \varphi _n (AB) -\varphi _n(A) \varphi _n(B) \| _{2,
\mathrm{tr}}^2
+ \frac{ \mathrm{Tr} (X_n^*X_n)}{k(n)+n}\\
& \le \| \varphi _n (AB) -\varphi _n(A) \varphi _n(B) \| _{2,
\mathrm{tr}}^2 + \frac{n \cdot \|A\|^2 \cdot \|B\|^2}{k(n)+n} \;.
\end{align*}
Using (\ref{eq:noverk(n)goesto0}) we get
\begin{equation*}
\lim_n\|\psi_n(AB)-\psi_n(A)\psi_n(B)\|_{2,\mathrm{tr}}=0.
\end{equation*}
On the other hand, for $A\in \cA$, we have
$$
\|A \| -\|\psi _n (A) \|  \le  \| A \| - \| P_n\pi (A)P_n\| \to 0
$$
so that (\ref{eq:norm}) holds for the sequence $(\psi _n )$. This
concludes the proof.
\end{proof}

For the next result recall that a representation $\pi$ of an
abstract C*-algebra $\cA$ on a Hilbert space $\cH$ is called {\em
essential} if $\pi(\cA)$ contains no nonzero compact operators.

\begin{theorem} \label{thm:characFolner} Let $\cA$ be a unital
separable C*-algebra. Then the following conditions are equivalent:
\begin{itemize}
 \item[(i)] There exists a faithful representation $\pi\colon\cA\to\cL(\cH)$ such that $\pi(\cA)$ has a
  F\o lner sequence.
 \item[(ii)] There exists a faithful essential representation $\pi\colon\cA\to\cL(\cH)$ such that $\pi(\cA)$ has a
  F\o lner sequence.
 \item[(iii)] Every faithful essential representation $\pi\colon\cA\to\cL(\cH)$ satisfies that $\pi(\cA)$ has a
 proper F\o lner sequence.
 \item[(iv)] There exists a non-zero representation $\pi\colon\cA\to\cL(\cH)$ such that $\pi(\cA)$ has an amenable trace.
 \item[(v)] Every faithful representation $\pi\colon\cA\to\cL(\cH)$ satisfies that $\pi(\cA)$ has an amenable trace.
 \item[(vi)] $\cA$ is a F\o lner C*-algebra.
\end{itemize}
 \end{theorem}
\begin{proof}
The implications (iii) $\Rightarrow$ (ii) $\Rightarrow$ (i) and (v)
$\Rightarrow$ (iv) are obvious. To show that (i) implies (ii)
suppose that $\pi\colon\cA\to\cL(\cH)$ is faithful and denote by
$\{P_n\}_n$ a F\o lner sequence for $\pi(\cA)$. Define the
representation
\[
 \widehat{\pi}\colon\cA\to\cL\left(\mathop{\oplus}\limits^\infty\cH\right)\;,\quad
 \widehat{\pi}(A):=\mathop{\oplus}\limits^\infty \pi(A)\;,\;A\in\cA\,,
\]
which, by construction, is essential. Moreover, choose
the sequence of finite-rank projections $\widehat{P}_n:=P_n\oplus0\oplus 0\dots$. Since
\[
 \lim_{n} \frac{\|\widehat{\pi}(A) \widehat{P}_n-\widehat{P}_n \widehat{\pi}(A)\|_2}{\|\widehat{P}_n\|_2} =
 \lim_{n} \frac{\|\pi(A) P_n-P_n \pi(A)\|_2}{\|P_n\|_2} = 0\;\;,\quad A\in\cA\;,
\]
we conclude that $\{\widehat{P}_n\}_n$ is a F\o lner sequence for $\widehat{\pi}(\cA)$.

The implication (ii) $\Rightarrow$ (iv) follows from
Proposition~\ref{exercise}~(i).

We now show that (iv) $\Rightarrow$ (v), following the proof of
\cite[Proposition 6.2.2]{bBrown08}: let
$\pi_0\colon\cA\to\cL(\cH_0)$ be a faithful representation and
identify $\cA$ with $\pi_0(\cA)$. Let $\pi\colon\cA\to\cL(\cH)$ be a
non-zero representation such that $\pi(\cA)$ has an amenable trace
$\tau$ which extends to a hypertrace $\psi$ on $\cL(\cH)$. From
Arveson's extension theorem (see, e.g., Theorem~1.6.1 in
\cite{bBrown08}), there exists a u.c.p.~map $\Phi\colon
\cL(\cH_0)\to\cL(\cH)$ extending $\pi$. Defining
$\psi_0:=\psi\circ\Phi$ it remains to show that $\psi_0$ is a
hypertrace on $\cL(\cH_0)$ (or
$\tau_0:=\psi_0 |_\cA$ is an amenable trace on
$\cA$). By construction it is immediate that $\psi_0$ is a state on
$\cL(\cH_0)$ extending the trace $\tau_0$. It remains to show that
$\psi_0$ is centralized by $\cA$: for any $X\in\cL(\cH_0)$ and
$A\in\cA$ we have
\begin{eqnarray*}
    \psi_0(AX)&=& \psi\left(\Phi(AX)\right) \;\;=\;\;\psi\big(\Phi(A)\Phi(X)\big)
                                            \;\;=\;\;\psi\big(\Phi(X)\Phi(A)\big) \\[2mm]
              &=&\psi_0(XA)\;,
\end{eqnarray*}
where for the second and fourth equalities we have used that $\cA$
is a multiplicative domain for $\Phi$.

(v) $\Rightarrow$ (vi) follows from \cite[Theorem 6.2.7]{bBrown08}.

(vi) $\Rightarrow $ (iii): Let $\iota\colon\cA\to\cL(\cH)$ be a
faithful essential representation, and identify $\cA$ with its image
$\iota (\cA)$ under $\iota $. Let $\varphi _n \colon \cA \to
M_{k(n)} (\C )$ be a sequence of u.c.p. maps such that $\lim _n k(n)
=\infty$ and such that (\ref{eq:mult-2}) holds. (Use the trick at
the beginning of the proof of Proposition \ref{prop:F=PF} to show
that we can always get such a sequence $k(n)$.)

By using the same arguments as in the proof of
Theorem~\ref{teo:apply} (disregarding the part concerning the approximation
of the amenable trace $\tau$),
we get that there is a proper F\o lner sequence $\{ P_n \}$
for $\cA$, as desired.
\end{proof}

\begin{remark}
 \begin{enumerate}

 \item[(i)] The class of C*-algebras introduced
in this section has been considered before by B\'edos. In
\cite{Bedos95} the author defines a C*-algebra $\cA$ to be {\it
weakly hypertracial} if $\cA$ has a non-degenerate representation
$\pi$ such that $\pi(\cA)$ has a hypertrace. In this sense, the
preceding theorem gives a new characterization of weakly
hypertracial C*-algebras in terms of u.c.p. maps.
\item[(ii)] The equivalences between (i), (iv) and (v)
in Theorem~\ref{thm:characFolner} are essentially known
(see \cite{Bedos95}).
\item[(iii)] The relation between amenable traces and the asymptotic 
multiplicativity condition needed in Definition~\ref{def:FA}~(i)
has been considered before (see, e.g., Theorem~2.3 in \cite{Brown04}).
\item[(iv)] The Toeplitz algebra (i.e., the C*-algebra generated by
the unilateral shift) is an example that shows that not every F\o lner 
algebra is a quasidiagonal C*-algebra.

\end{enumerate}
\end{remark}

In the final part of this section we recall some operator 
algebraic properties of the class of F\o lner C*-algebras.
Most of them have already been proved in Section~2 of \cite{Bedos95} in a more general context.
For convenience of the reader and to make this exposition partly
self-contained we give short proofs in some cases.

\begin{corollary}
\label{cor:nonzeroquotients} Let $\cA$ be a unital separable
C*-algebra. If a nonzero quotient of $\cA$ is a F\o lner C*-algebra,
then $\cA$ is a F\o lner C*-algebra. In particular, any C*-algebra
admitting a finite-dimensional representation is a F\o lner
C*-algebra.
\end{corollary}

\begin{proof}
This follows from condition (iv) in Theorem~\ref{thm:characFolner}.
Indeed, let $\rho \colon \cA \to \cB$ be a surjective
*-homomorphism onto a F\o lner C*-algebra $\cB$, and let $\pi \colon
\cB\to \cL (\cH )$ be a nonzero representation such that $\pi (\cB
)$ has an amenable trace (Theorem~\ref{thm:characFolner}(iv)). Then
$\pi \circ \rho $ is a nonzero representation of $\cA$ such that
$\pi\circ \rho (\cA )$ has an amenable trace. By
Theorem~\ref{thm:characFolner}~(iv) we conclude that $\cA$ is a F\o lner C*-algebra.
\end{proof}

Let $\cA\subset \cL (\cH)$ be a F\o lner C*-algebra. Note that it
can happen that $\cA$ has no {\em proper} F\o lner sequence (recall
Definition~\ref{def:Foelner}~(i)) as the following simple example
shows: let $\cB\subset \cL(\cH_0)$ be a unital separable C*-algebra
which is not a F\o lner C*-algebra acting on an infinite dimensional
Hilbert space and define $\cA:=\C \oplus\cB$ on
$\cH:=\C\oplus\cH_0$. By the previous corollary $\cA$ is a F\o lner
C*-algebra, and it is readily checked that $\cA$ has no proper F\o
lner sequence in $\cL (\cH)$ (although, by Theorem
\ref{thm:characFolner}(iii), it will have a proper F\o lner sequence
in a different representation).

\medskip

For the next result we recall some standard notation. We will denote
by $\cA \odot \cB$ the algebraic tensor product of two C*-algebras
$\cA$ and $\cB$, and by $\cA \otimes \cB$ its minimal tensor
product. The fact that any C*-tensor product of two F\o lner
C*-algebras is a F\o lner C*-algebra was proved by B\'edos in
\cite[Proposition 2.13]{Bedos95}. However the nice interplay between
amenable traces and F\o lner sequences shown in our proof is a
genuine application of our approach.

\begin{proposition}
\label{prop:tensorproduct} Let $\cA$ and $\cB$ be two F\o lner
C*-algebras, and let $\pi _A\colon \cA \to \cL (\cH_A)$ and $\pi_B
\colon \cB \to \cL (\cH_B) $ be faithful essential representations
of $\cA$ and $\cB$. Then $\cA \otimes \cB$ is a F\o lner C*-algebra.
Moreover if $\tau _A$ and $\tau _B$ are amenable traces on $\cA$ and
$\cB$, then there exists a hypertrace on $\cL (\cH _A\otimes \cH
_B)$ extending the state $\tau _A\otimes \tau_B$ on $\cA\otimes
\cB$.
\end{proposition}

\begin{proof}
By Proposition \ref{exercise}(ii), there are F\o lner sequences
$\{P_n\}_n$ and $\{ Q_n\}_n$ for $\cA$ and $\cB$, acting on $\cH_A$
and $\cH_B$ respectively, such that
$$\tau_{\cA}(A)=\lim_{n\to\infty} \frac{\mathrm{Tr}(AP_n)}{\mathrm{Tr}(P_n)},\quad A\in\cA;
\qquad \tau_{\cB}(B)=\lim_{n\to\infty}
\frac{\mathrm{Tr}(BQ_n)}{\mathrm{Tr}(Q_n)},\quad B\in\cB . $$ We
show that $\{ P_n\otimes Q_n\}_n$ is a F\o lner sequence for
$\cA\otimes \cB\subset \cL (\cH_A\otimes \cH_B)$. Let $A\in \cA$ and
$B\in \cB$. Then we have
\begin{align*}
&\frac{\| (\1 \otimes \1-P_n\otimes Q_n) (A\otimes B) (P_n\otimes
Q_n)\|_2^2}{\| P_n\otimes Q_n \|_2^2 }\\[2mm]
& \le \frac{\| ((\1-P_n)\otimes \1 )(A\otimes B)(P_n\otimes
Q_n)\|_2^2}{\| P_n\otimes Q_n \|_2^2 } + \frac{\|(P_n\otimes
(\1-Q_n))(A\otimes B)(P_n\otimes Q_n) \|_2^2}{\| P_n\otimes Q_n
\|_2^2 } \\[2mm]
& = \frac{\| (\1-P_n)AP_n \|_2^2}{\| P_n\|_2^2}\cdot
\frac{\|BQ_n\|_2^2}{\| Q_n\|_2^2} + \frac{\|P_nAP_n\|_2^2}{\|
P_n\|_2^2}\cdot \frac{\|(\1-Q_n)BQ_n\|_2 ^2}{\| Q_n\|_2^2}\\[2mm]
& \le \|B\|^2\cdot  \frac{\| (\1-P_n)AP_n \|_2^2}{\| P_n\|_2^2}+
\|A\|^2 \cdot \frac{\|(\1-Q_n)BQ_n\|_2 ^2}{\| Q_n\|_2^2}
\mathop{\longrightarrow}\limits_{n\to\infty} 0.
\end{align*}
Since the set $\{A\otimes B \mid A\in \cA, B\in \cB \}$ is a
selfadjoint, generating set for $\cA \otimes \cB$, it follows from
Proposition \ref{total} that $\{P_n\otimes Q_n\}$ is a F\o lner
sequence for $\cA\otimes \cB$. This shows that $\cA\otimes \cB$ is a
F\o lner C*-algebra.

Let $\psi $ be a weak-* cluster point of the sequence of states
$\{\psi _n \}_n$ defined by
$$\psi _n (X)=\frac{\mathrm{Tr}(X(P_n\otimes Q_n))}{\mathrm{Tr}(P_n\otimes Q_n)}\;,\quad X\in\cL(\cH_A\otimes \cH_B)\;.
$$
Then $\psi (A\otimes B)= \tau_{\cA}(A)\tau_{\cB}(B)$ for all $A\in
\cA$ and $B\in \cB$. Hence, $\psi$ extends the state
$\tau_{\cA}\otimes \tau_{\cB}$ on $\cA \otimes \cB$. Moreover,
$\psi$ is a hypertrace for $\cA\otimes \cB$ on $\cL (\cH_{\cA}
\otimes \cH_{\cB})$ (see the proof of Proposition
\ref{exercise}(i)). This concludes the proof.
\end{proof}

\begin{corollary}\cite[Proposition 2.13]{Bedos95}
\label{cor:alphatensorproduct}  Let $\cA$ and $\cB$ be two F\o lner
C*-algebras, and let $\alpha$ be any C*-norm on the algebraic tensor
product $\cA \odot \cB$. Then $\cA \otimes _{\alpha} \cB$ is a F\o
lner C*-algebra.
\end{corollary}

\begin{proof}
Let $\alpha$ be any C*-norm on the algebraic tensor product $\cA
\odot \cB$. Then there is a surjective $*$-homomorphism $\cA\otimes
_{\alpha} \cB\to \cA \otimes \cB$. So the result follows from
Proposition \ref{prop:tensorproduct} and Corollary
\ref{cor:nonzeroquotients}.
\end{proof}

The relation with nuclearity is as follows. Recall that there are
non-nuclear F\o lner C*-algebras, such as $C^* (\mathbb F_2)$, the
full C*-algebra of the free group on two generators, which is even
quasidiagonal.

\begin{corollary}
\label{cor:nuclversusFol} Let $\cA$ be a unital  nuclear C*-algebra.
Then $\cA$ is a F\o lner C*-algebra if and only if $\cA$ admits a
tracial state. In particular, every stably finite unital nuclear
C*-algebra is F\o lner.
\end{corollary}

\begin{proof}
The first part follows from Theorem \ref{thm:characFolner} and
\cite[Proposition 6.3.4]{bBrown08}. If $\cA$ is a stably finite
unital nuclear C*-algebra then $\cA$ admits a tracial state by
\cite[Corollary V.2.1.16]{BlacEnc}.
\end{proof}

Note that the Cuntz algebras $\mathcal O _n$ are nuclear but not
F\o lner.

Finally, we characterize F\o lner reduced crossed products. The
proof of the following result follows from Proposition~2.12 in
\cite{Bedos95}. Let us remark that it is possible to give a
variation of B\'edos' proof using Day's fixed point theorem
(cf., \cite{Day61}).

\begin{proposition}
\label{thm:crossedprods} Let $\Gamma $ be a countable discrete group
and let $\alpha $ be an action of $\Gamma $ on a separable
C*-algebra $\cA$. Then the following conditions are equivalent:
\begin{enumerate}
\item[(i)] $\cA \rtimes _{\alpha,r}\Gamma $ is a F\o lner C*-algebra.
\item[(ii)] $\Gamma $ is amenable and $\cA$ has a $\Gamma
$-invariant amenable trace.
\item[(iii)] $\cA$ is a F\o lner C*-algebra and $\Gamma $ is an
amenable group.
\end{enumerate}
\end{proposition}

\paragraph{\bf Acknowledgements:}
This article was partly written while the authors were visiting the
{\em Centre de Recerca Matem\`atica} (Barcelona). It is a pleasure to
thank useful conversations with Nate Brown during this visit.
We also want to thank Erik B\'edos for his friendly interest in this article and
for many suggestions to improve it.

\providecommand{\bysame}{\leavevmode\hbox to3em{\hrulefill}\thinspace}

\end{document}